\documentclass{article}
\usepackage{amssymb,amsmath,xypic}

\title{An embedding theorem for Hilbert categories}
\author{Chris Heunen}

\newcommand{\after}{\circ}
\newcommand{\cat}[1]{\ensuremath{\mathbf{#1}}}
\newcommand{\Cat}[1]{\ensuremath{\mathbf{#1}}}
\newcommand{\id}[1][]{\ensuremath{\mathrm{id}_{#1}}}
\newcommand{\I}{\ensuremath{I}}
\newcommand{\op}{\ensuremath{^{\mathrm{op}}}}
\newcommand{\field}[1]{\ensuremath{\mathbb{#1}}}
\newcommand{\inprod}[2]{\ensuremath{\langle #1\,|\,#2 \rangle}}
\newcommand{\tensor}{\ensuremath{\otimes}}

\newcommand{\ie}{\textit{i.e.}~}
\newcommand{\eg}{\textit{e.g.}~}
\newcommand{\coker}{\ensuremath{\mathrm{coker}}}
\renewcommand{\Im}{\ensuremath{\mathrm{Im}}}
\newcommand{\Sub}{\ensuremath{\mathrm{Sub}}}
\newcommand{\SMod}[1][]{\ensuremath{\Cat{SMod}_{#1}}}
\newcommand{\HMod}[1][]{\ensuremath{\Cat{HMod}_{#1}}}
\newcommand{\sHMod}[1][]{\ensuremath{\Cat{sHMod}_{#1}}}
\newcommand{\Hilb}{\Cat{Hilb}}
\newcommand{\preHilb}[1][]{\Cat{preHilb}_{#1}}
\makeatother
 \newdir{ >}{{}*!/-7.5pt/@{>}}
 \newdir{|>}{!/4.5pt/@{|}*:(1,-.2)@^{>}*:(1,+.2)@_{>}}
 \newdir{ |>}{{}*!/-3pt/@{|}*!/-7.5pt/:(1,-.2)@^{>}*!/-7.5pt/:(1,+.2)@_{>}}
\makeatletter
\newtheorem{theorem}{Theorem}
\newtheorem{lemma}{Lemma}
\newtheorem{proposition}{Proposition}
\newtheorem{corollary}{Corollary}
\newtheorem{definition}{Definition}

\newenvironment{proof}[1][Proof]%
{ \begin{trivlist}%
  \item[\hskip \labelsep {\bfseries #1}]%
}%
{ \end{trivlist}%
}
\newcommand{\qed}{\nobreak\hfill$\Box$}

\begin{document}
\maketitle

\begin{abstract}
  We axiomatically define (pre-)Hilbert categories. The axioms
  resemble those for monoidal Abelian categories with the addition of
  an involutive functor. We then prove embedding theorems: any locally
  small pre-Hilbert category whose monoidal unit is a simple generator
  embeds (weakly) monoidally into the category of pre-Hilbert spaces and
  adjointable maps, preserving adjoint morphisms and all finite
  (co)limits. An intermediate result that is important in its own
  right is that the scalars in such a category necessarily form an
  involutive field. In case of a Hilbert category, the embedding
  extends to the category of Hilbert spaces and continuous linear maps.
  The axioms for (pre-)Hilbert categories are weaker than the axioms found
  in other approaches to axiomatizing 2-Hilbert spaces. Neither enrichment
  nor a complex base field is presupposed. A comparison to other
  approaches will be made in the introduction. 
\end{abstract}

\section{Introduction}
\label{sec:introduction}

Modules over a ring are fundamental to algebra. 
Distilling their categorical properties results in the definition of
Abelian categories, which play a prominent part in algebraic geometry,
cohomology and pure category theory. The prototypical Abelian category
is that of modules over a fixed ring. Indeed, Mitchell's famous
embedding theorem states that any small Abelian category embeds into
the category of modules over some ring~\cite{mitchell:categories,
freyd:abelian}.  

Likewise, the category $\Hilb$ of (complex) Hilbert spaces and
continuous linear transformations is of paramount importance in
quantum theory and functional analysis. So is the category $\preHilb$
of (complex) pre-Hilbert spaces and adjointable maps. Although they
closely resemble the category of modules (over the complex field),
neither $\Hilb$ nor $\preHilb$ is Abelian. At the heart of the
failure of $\Hilb$ and $\preHilb$ to be Abelian is the existence of a
functor providing adjoint morphisms, called a dagger, that witnesses
self-duality. Hence the proof method of Mitchell's embedding theorem
does not apply.   

This article evens the situation, by combining ideas from Abelian
categories and dagger categories. The latter have been used fruitfully
in modeling aspects of quantum physics
recently~\cite{abramskycoecke:quantum, selinger:completelypositive,
selinger:daggeridempotents}. 
We axiomatically define \emph{\mbox{(pre-)Hilbert} categories}. The axioms
closely resemble those of a monoidal Abelian category, with the
addition of a dagger. Their names are justified by proving appropriate
embedding theorems: roughly speaking, pre-Hilbert categories embed
into $\preHilb$, and Hilbert categories embed into $\Hilb$. These
embeddings are in general not full, and only weakly monoidal. But
otherwise they preserve all the structure of pre-Hilbert categories,
including all finite (co)limits, and adjoint morphisms (up to an
isomorphism of the induced base field).

To sketch the historical context of these embedding theorems, let us
start by recalling that a category is called Abelian when: 
\begin{enumerate}
  \item it has finite biproducts;
  \item it has (finite) equalisers and coequalisers;
  \item every monomorphism is a kernel, and every epimorphism is a cokernel.
\end{enumerate}
We can point out already that Definition~\ref{def:hilblikecategory}
below, of pre-Hilbert category, is remarkably similar, 
except for the occurence of a dagger.
From the above axioms, enrichment over Abelian groups follows.
For the Abelian embedding theorem, there are (at least) two
`different' proofs, one by Mitchell~\cite{mitchell:categories}, and
one by Lubkin~\cite{lubkin}. Both operate by first embedding into 
the category $\Cat{Ab}$ of Abelian groups, and then adding a scalar
multiplication. This approach can be extended to also take tensor
products into account~\cite{hai:monoidalabelian}. However, as
$\Cat{Ab}$ is not a self-dual category, this strategy does not extend
straightforwardly to the setting of Hilbert spaces. 

Several authors have used an involution on the given category 
in this context before. Specifically, by a \emph{dagger} on a category
$\cat{C}$ we mean a functor $\dag \colon \cat{C}\op \to \cat{C}$ that
satisfies $X^\dag=X$ on objects and $f^{\dag\dag}=f$ on morphisms.
For example, \cite[Proposition~1.14]{ghezlimaroberts:wstarcategories} 
proves that any \emph{C*-category} embeds into $\Cat{Hilb}$.
Here, a C*-category is a category such that:
\begin{enumerate}
  \item it is enriched over complex Banach spaces and linear
    contractions; 
  \item it has an antilinear dagger;
  \item every $f \colon X \to Y$ satisfies $f^\dag f=0 \Rightarrow
    f=0$, \\ and there is a $g \colon X \to X$
    with $f^\dag f = g^\dag g$;
  \item $\|f\|^2 = \|f^\dag f\|$ for every morphism $f$.
\end{enumerate}
The embedding of a C*-category into $\Cat{Hilb}$ uses powerful
analytical methods, as it is basically an extension of the
Gelfand-Naimark theorem~\cite{gelfandnaimark} showing that every
C*-algebra (\ie one-object C*-category) can be realized concretely
as an algebra of operators on a Hilbert space. Compare the previous
definition to Definition~\ref{def:hilblikecategory} below: the axioms
of (pre-)Hilbert categories are much weaker. For example, nothing about
the base field is built into the definition. In fact, one of our main
results derives the fact that the base semiring is a field. For the
same reason, our situation also differs from \emph{Tannakian
  categories}~\cite{deligne:tannakiancategories}, that are otherwise
somewhat similar to our \mbox{(pre-)Hilbert} categories.  
Moreover, (pre-)Hilbert categories do not presuppose any enrichment, but
derive it from prior principles. 

A related embedding theorem is~\cite{doplicherroberts:duality} (see
also \cite{mueger:duality} for a categorical account). It
characterizes categories that are equivalent to the category 
of finite-dimensional unitary representations of a uniquely determined
compact supergroup. Without explaining the postulates, let us mention
that the categories $\cat{C}$ considered:
\begin{enumerate}
  \item are enriched over complex vector spaces;
  \item have an antilinear dagger;
  \item have finite biproducts;
  \item have tensor products $(\I,\tensor)$;
  \item satisfy $\cat{C}(\I,\I) \cong \field{C}$;
  \item every projection dagger splits;
  \item every object is compact.
\end{enumerate}
Our definition of (pre-)Hilbert category also requires
2,3, and 4 above. Furthermore, we will also use an analogue of 5,
namely that $\I$ is a \emph{simple generator}. But notice, again, that 1
above presupposes a base field $\field{C}$, and enrichment over
complex vector spaces, whereas (pre-)Hilbert categories do not. 
As will become clear, our definition and theorems function regardless
of dimension; we will come back to dimensionality and the compact
objects in 7 above in Subsection~\ref{subsec:dimension}.   

This is taken a step further by \cite{baez:twohilbertspaces}, which
follows the ``categorification'' programme originating in homotopy
theory~\cite{kapranovvoevodsky}. A \emph{2-Hilbert space} is a
category that:
\begin{enumerate}
  \item is enriched over $\Cat{Hilb}$;
  \item has an antilinear dagger;
  \item is Abelian;
\end{enumerate}
The category $\Cat{2Hilb}$ of 2-Hilbert spaces turns out to be
monoidal. Hence it makes sense to define a \emph{symmetric
2-H*-algebra} as a commutative monoid in $\Cat{2Hilb}$, in which
furthermore every object is compact. Then,
\cite{baez:twohilbertspaces} proves that every symmetric 2-H*-algebra
is equivalent to a category of continuous unitary finite-dimensional
representations of some compact supergroupoid. Again, the proof is
basically a categorification of the Gelfand-Naimark theorem.
Although the motivation for 2-Hilbert spaces is a categorification of
a single Hilbert space, they resemble our (pre-)Hilbert categories,
that could be seen as a characterisation of the category of all
Hilbert spaces. However, there are important differences. First of
all, axiom 1 above again presupposes both the complex numbers as a
base field, and a nontrivial enrichment. For example, as (pre-)Hilbert
categories assume no enrichment, we do not have to consider coherence
with conjugation. Moreover, \cite{baez:twohilbertspaces} considers
only finite dimensions, whereas the category of all Hilbert spaces, 
regardless of dimension, is a prime example of a (pre-)Hilbert
category (see also Subsection~\ref{subsec:dimension}). Finally, a
2-Hilbert space is an Abelian category, whereas a (pre-)Hilbert
category need not be (see Appendix~\ref{sec:thecategoryHilb}). 

Having sketched how the present work differs from existing work,
let us end this introduction by making our approach a bit more precise
while describing the structure of this
paper. Section~\ref{sec:hilbertcategories} introduces our 
axiomatisation. We then embark on proving embedding theorems for
such categories $\cat{H}$, under the assumption that the monoidal unit
$\I$ is a generator. First, we establish a functor $\cat{H} \to
\sHMod[S]$, embedding 
$\cat{H}$ into the category of strict \emph{Hilbert semimodules} over
the involutive semiring
$S=\cat{H}(\I,\I)$. Section~\ref{sec:Hilbertsemimodules} deals with this
rigorously. This extends previous work, that shows that a category
$\cat{H}$ with just biproducts and tensor products is enriched over
$S$-semimodules~\cite{heunen:semimoduleenrichment}.  
If moreover $\I$ is simple, Section~\ref{sec:scalarfield} 
proves that $S$ is an involutive field of characteristic
zero. This is an improvement over \cite{vicary:complexnumbers}, on which
Section~\ref{sec:scalarfield} draws for inspiration. Hence
$\sHMod[S]=\preHilb[S]$, and $S$ embeds into a field
isomorphic to the complex numbers. \emph{Extension of scalars} gives
an embedding $\preHilb[S] \to \preHilb[\field{C}]$, discussed in
Section~\ref{sec:Extensionofscalars}.  
Finally, when $\cat{H}$ is a Hilbert category,
Section~\ref{sec:completion} shows that Cauchy completion induces an
embedding into $\Hilb$ of the image of $\cat{H}$ in
$\preHilb$. Composing these functors then gives an embedding $\cat{H}
\to \Hilb$. Along the way, we also discuss how a great deal of the
structure of $\cat{H}$ is preserved under this embedding: in addition to
being (weakly) monoidal, the embedding preserves all finite limits and
colimits, and preserves adjoint morphisms up to an isomorphism of the
complex field.  Section~\ref{sec:conclusion} concludes the main body
of the paper, and Appendix~\ref{sec:thecategoryHilb} considers
relevant aspects of the category $\Hilb$ itself.

\section{(Pre-)Hilbert categories}
\label{sec:hilbertcategories}

This section introduces the object of study. Let $\cat{H}$ be a category.
A functor $\dag\colon \cat{H}\op \to \cat{H}$ with $X^\dag=X$ on objects and 
$f^{\dag\dag}=f$ on morphisms is called a \emph{dagger}; the
pair $(\cat{H}, \dag)$ is then called a \emph{dagger category}. Such
categories are automatically isomorphic to their opposite.
We can consider coherence of the dagger with respect to all
sorts of structures. For example, a morphism $m$ in such a category
that satisfies $m^\dag m = \id$ is 
called a \emph{dagger mono(morphism)} and is denoted $\xymatrix@1{\ar@{ |>->}[r] &
}$. Likewise, $e$ is a \emph{dagger epi(morphism)}, denoted
$\xymatrix@1{\ar@{-|>}[r] & }$, when $ee^\dag=\id$. A morphism is
called a \emph{dagger isomorphism} when it is both dagger epic and
dagger monic. Similarly, a biproduct
on such a category is called a \emph{dagger biproduct} when
$\pi^\dag = \kappa$, where $\pi$ is a projection and $\kappa$ an
injection. This is equivalent to demanding $(f \oplus g)^\dag = f^\dag
\oplus g^\dag$.
Also, an equaliser is called a \emph{dagger equaliser} when
it can be represented by a dagger mono, and a kernel is called a
\emph{dagger kernel} when it can be represented by a dagger
mono. Finally, a dagger category $\cat{H}$ is 
called \emph{dagger monoidal} when it is 
equipped with monoidal structure $(\tensor,\I)$ that is compatible with
the dagger, in the sense that $(f \tensor g)^\dag = f^\dag
\tensor g^\dag$, and the coherence isomorphisms are
dagger isomorphisms. 

\begin{definition}
\label{def:hilblikecategory}
  A category is called a \emph{pre-Hilbert category} when:
  \begin{itemize}
    \item it has a dagger;
    \item it has finite dagger biproducts;
    \item it has (finite) dagger equalisers;
    \item every dagger mono is a dagger kernel; 
    \item it is symmetric dagger monoidal.
  \end{itemize}
\end{definition}

Notice that no enrichment of any kind is assumed. Instead, it will
follow. Also, no mention is made of the complex numbers or any other
base field. This is a notable difference with other approaches
mentioned in the Introduction.

Our main theorem will assume that the monoidal unit $\I$ is a
\emph{generator}, \ie that $f=g\colon X \to Y$ when $fx=gx$ for all $x\colon \I
\to X$. A final condition we will use is the following: the monoidal
unit $\I$ is called \emph{simple} when $\Sub(\I) = \{0,\I\}$ and
$\cat{H}(\I,\I)$ is at most of continuum cardinality. 
Intuitively, a simple object $\I$ can be thought of as being
``one-dimensional''. The definition of a simple object in abstract
algebra is usually given without the size requirement, which we
require to ensure that the induced base field is not too large. With an
eye toward future generalisation, this paper postpones assuming $\I$
simple as long as possible. 

The category $\Cat{Hilb}$ itself is a locally small pre-Hilbert
category whose monoidal unit is a simple generator, and so is its
subcategory $\Cat{fdHilb}$ of finite-dimensional Hilbert spaces (see
Appendix~\ref{sec:thecategoryHilb}).   

Finally, a pre-Hilbert category whose morphisms are bounded is called
a \emph{Hilbert category}. It is easier to define this last axiom
rigorously after a discussion of scalars, and so we defer this to
Section~\ref{sec:completion}.

\section{Hilbert semimodules}
\label{sec:Hilbertsemimodules}

In this section, we study Hilbert semimodules, to be defined in
Definition~\ref{def:hilbertsemimodule} below. It turns out that the
structure of a pre-Hilbert category $\cat{H}$ gives rise to an
embedding of $\cat{H}$ into a category of Hilbert semimodules.
Let us first recall the notions of semiring and semimodule in some
detail, as they might be unfamiliar to the reader.

A \emph{semiring} is roughly a ``ring that does not necessarily have
subtraction''. All the semirings we use will be commutative.
Explicitly, a commutative semiring consists of a set $S$, two
elements $0,1 \in S$, and two binary operations $+$ and $\cdot$ on
$S$, such that the following equations hold for all $r,s,t \in S$:
\begin{align*}
  0 + s & = s, &
  1 \cdot s & = s, \\
  r + s & = s + r, &
  r \cdot s & = s \cdot r, \\
  r + (s + t) & = (r + s) + t,  &
  r \cdot (s \cdot t) & = (r \cdot s) \cdot t, \\
  s \cdot 0 & = 0, &
  r \cdot (s + t) & = r \cdot s + r \cdot t.
\end{align*}
Semirings are also known
as \emph{rigs}. For more information we refer to~\cite{golan:semirings}. 

A \emph{semimodule} over a
commutative semiring is a generalisation of a module over a
commutative ring, which in turn is a generalisation of a vector space
over a field. Explicitly, a semimodule over a commutative semiring $S$
is a set $M$ with a specified element $0 \in M$, equipped with
functions $+ \colon M \times M \to M$ and $\cdot \colon S \times M \to
M$ satisfying the following equations  
for all $r,s \in S$ and $l,m,n \in M$:
\begin{align*}
  s \cdot (m + n) & = s \cdot m + s \cdot n, &
  0 + m & = m, \\ 
  (r + s) \cdot m & = r \cdot m + s \cdot m, &
  m + n & = n + m, \\
  (r \cdot s) \cdot m & = r \cdot (s \cdot m), &
  l + (m + n) & = (l + m) + n, \\
  0 \cdot m & = 0, & 
  1 \cdot m & = m, \\
  s \cdot 0 & = 0.
\end{align*}
A function between $S$-semimodules is called $S$-semilinear when it
preserves $+$ and $\cdot$. Semimodules over a commutative
semiring $S$ and $S$-semilinear transformations form a category
$\SMod[S]$ that largely behaves like that of modules over a
commutative ring. For example, it is symmetric monoidal closed. 
The tensor product of $S$-semimodules $M$ and $N$ is generated by
elements of the form $m \tensor n$ for $m \in M$ and $n \in N$,
subject to the following relations:
\begin{align*}
  (m+m') \tensor n & = m \tensor n + m '\tensor n, \\
  m \tensor (n+n') & = m \tensor n + m \tensor n', \\
  (s\cdot m) \tensor n & = m \tensor (s \cdot n), \\
  k \cdot (m \tensor n) & = (k \cdot m) \tensor n = m \tensor (k \cdot
  n), \\
  0 \tensor n & = 0 = m \tensor 0,
\end{align*}
for $m,m' \in M$, $n, n' \in N$, $s \in S$ and $k \in \field{N}$.
It satisfies a universal property that differs slightly from that of
modules over a ring: every function from $M \times N$ to a commutative
monoid $T$ that is semilinear in both variables separately factors
uniquely through a semilinear function from $M \tensor N$ to
$T\slash\mathop{\sim}$, where $t \sim t'$ iff there is a $t'' \in T$ with
$t+t''=t'+t''$. 
For more information about semimodules, we refer to~\cite{golan:semirings},
or~\cite{heunen:semimoduleenrichment} for a categorical perspective. 

A commutative \emph{involutive semiring} is a commutative semiring
$S$ equipped with a semilinear involution $\ddag \colon  S \to S$. An
element $s$ of an involutive semiring is called \emph{positive},
denoted $s\geq 0$, when it is of the form $s=t^\ddag t$. The set of
all positive elements of an involutive semiring $S$ is denoted $S^+$.  
For every semimodule $M$ over a commutative involutive semiring, there is
also a semimodule $M^\ddag$, whose carrier set and addition are
the same as before, but whose scalar multiplication $sm$
is defined in terms of the scalar multiplication of $M$ by $s^\ddag
m$. An $S$-semilinear
map $f\colon M \to N$ also induces a map $f^\ddag\colon M^\ddag \to
N^\ddag$ by $f^\ddag(m)=f(m)$. Thus, an involution $\ddag$ on a
commutative semiring $S$ induces an involutive functor $\ddag \colon 
\SMod[S] \to \SMod[S]$. 

Now, just as pre-Hilbert spaces are vector spaces equipped with an
inner product, we can consider semimodules with an inner product.

\begin{definition}
\label{def:hilbertsemimodule}
  Let $S$ be a commutative involutive semiring.
  An $S$-semimodule $M$ is called a \emph{Hilbert semimodule} when
  it is equipped with a morphism $\inprod{-}{-}\colon M^\ddag \tensor M
  \to S$ of $\SMod[S]$, satisfying
  \begin{itemize}
    \item $\inprod{m}{n} = \inprod{n}{m}^\ddag$,
    \item $\inprod{m}{m} \geq 0$, and
    \item $\inprod{m}{-} = \inprod{n}{-} \Rightarrow m=n$.
  \end{itemize}
  The Hilbert semimodule is called \emph{strict} if moreover
  \begin{itemize}
    \item $\inprod{m}{m} = 0 \Rightarrow m=0$.
  \end{itemize}
\end{definition}

For example, $S$ itself is a Hilbert $S$-semimodule by
$\inprod{s}{t}_S = s^\ddag t$. Recall that a semiring $S$ is
\emph{multiplicatively cancellative} when $sr=st$ and $s \neq 0$ imply
$r=t$~\cite{golan:semirings}. Thus $S$ is a strict Hilbert
$S$-semimodule iff $S$ is multiplicatively cancellative. 

The following choice of morphisms is also the standard choice of
morphisms between Hilbert
C*-modules~\cite{lance:hilbertcstarmodules}.\footnote{There is another
analogy for this choice of morphisms. Writing $M^*=M \multimap S$ for
the dual $S$-semimodule of $M$,
Definition~\ref{def:hilbertsemimodule} resembles that of a `diagonal'
object of the Chu construction on $\SMod[S]$. The Chu construction
provides a `generalised topology', like an inner product provides a
vector space provides with a metric and hence a
topology~\cite{barr:starautonomousoncemore}.}

\begin{definition}
  A semimodule homomorphism $f\colon M \to N$ between Hilbert
  $S$-semi\-modules is called \emph{adjointable} when there is a
  semimodule homomorphism $f^\dag\colon N \to M$ such that
  $\inprod{f^\ddag(m)}{n}_N = \inprod{m}{f^\dag(n)}_M$ for all $m
  \in M^\ddag$ and $n \in N$.    
\end{definition}

The \emph{adjoint} $f^\dag$ is unique since the power transpose of the
inner product is a monomorphism. However, it does not necessarily 
exist, except in special situations like (complete) Hilbert spaces
($S=\field{C}$ or $S=\field{R})$ and bounded semilattices
($S$ is the Boolean semiring $\field{B}=(\{0,1\},\max,\min)$,
see~\cite{paseka:hilbertquantales}). Hilbert 
$S$-semimodules and adjointable maps organise themselves in a category
$\HMod[S]$. We denote by $\sHMod[S]$ the full subcategory of strict
Hilbert $S$-semimodules. The choice of morphisms ensures that
$\HMod[S]$ and $\sHMod[S]$ are dagger categories. Let us study some of
their properties. The following lemma could be regarded as an
analogue of the Riesz-Fischer theorem
\cite[Theorem~III.1]{reedsimon:functionalanalysis}.

\begin{lemma}
\label{lem:hmodissmodenriched}
  $\HMod[S]$ is enriched over $\SMod[S]$, and 
  \[
    \HMod[S](S,X) = \SMod[S](S,X) \cong X,
  \]
  where we suppressed the forgetful functor $\HMod[S] \to \SMod[S]$.
\end{lemma}
\begin{proof}
  For $X,Y \in \HMod[S]$, the zero map $X \to Y$ in $\SMod[S]$ is
  self-adjoint, and hence a morphism in $\HMod[S]$. 
  If $f,g\colon X \to Y$ are adjointable, then so is $f+g$, as its adjoint
  is $f^\dag + g^\dag$.
%   \begin{align*}
%         \inprod{(f+g)(x)}{y}_Y 
%     & = \inprod{f(x)}{y}_Y + \inprod{g(x)}{y}_Y \\
%     & = \inprod{x}{f^\dag(y)}_X + \inprod{x}{g^\dag(y)}_X \\
%     & = \inprod{x}{(f^\dag+g^\dag)(y)}_X.
%   \end{align*}
  If $s \in S$ and $f\colon X \to Y$ is adjointable, then so is $sf$, as its
  adjoint is $s^\ddag f^\dag$:
  \[
      \inprod{sf(x)}{y}_Y
    = s^\ddag \inprod{f(x)}{y}_Y 
    = s^\ddag \inprod{x}{f^\dag(y)}_X 
    = \inprod{x}{s^\ddag f^\dag(y)}_X.
  \]
  Since composition is bilinear, $\HMod[S]$ is thus enriched over
  $\SMod[S]$. 

  Suppose $X \in \HMod[S]$, and $f\colon S \to X$ is a morphism of
  $\SMod[S]$. Define a morphism $f^\dag\colon X \to S$ of $\SMod[S]$ by $f^\dag 
  = \inprod{f(1)}{-}_X$. Then
  \[
      \inprod{f(s)}{x}_X = \inprod{sf(1)}{x}_X 
    = s^\ddag\inprod{f(1)}{x}_X  
    = s^\ddag f^\dag(x) = \inprod{s}{f^\dag(x)}_S.
  \]
  Hence $f \in \HMod[S](S,X)$. Obviously $\HMod[S](S,X)
  \subseteq \SMod[S](S,X)$. The fact that $S$ is a generator for
  $\HMod[S]$ proves the last claim $\SMod[S](S,X) \cong X$.
\end{proof}

Notice from the proof of the above lemma that the inner product of $X$
can be reconstructed from $\HMod[S](S,X)$. Indeed, if we temporarily define
$\underline{x}\colon S \to X$ by $1 \mapsto x$ for $x \in X$, then we can use the
adjoint by
\[
  \inprod{x}{y}_X = \inprod{\underline{x}(1)}{y}_X
  = \inprod{1}{\underline{x}^\dag(y)}_S = 
  \underline{x}^\dag (y) = \underline{x}^\dag \after
  \underline{y}(1). 
\]
We can go further by providing $\HMod[S](S,X)$ itself with
the structure of a Hilbert $S$-semimodule: for $f,g \in
\HMod[S](S,X)$, put $\inprod{f}{g}_{\HMod[S](S,X)} = f^\dag \after g
(1)$. Then the above lemma can be strengthened as follows.

\begin{lemma}
\label{lem:HModSXisX}
  There is a dagger isomorphism $X \cong \HMod[S](S,X)$ in $\HMod[S]$.
\end{lemma}
\begin{proof}
  Define $f\colon X \to \HMod[S](S,X)$ by $f(x)=x\cdot(-)$, and
  $g\colon \HMod[S](S,X) \to X$ by $g(\varphi) = \varphi(1)$. Then $f \after
  g=\id$ and $g\after f = \id$, and moreover $f^\dag=g$:
  \begin{align*}
        \inprod{x}{g(\varphi)}_X 
    & = \inprod{x}{\varphi(1)}_X
      = (x \cdot (-))^\dag \after \varphi(1) \\
    & = \inprod{x \cdot (-)}{\varphi}_{\HMod[S](S,X)}
      = \inprod{f(x)}{\varphi}_{\HMod[S](S,X)}
  \end{align*}
\end{proof}

Recall that (a subset of) a semiring is called
\emph{zerosumfree} when $s+t=0$ implies $s=t=0$ for all elements $s$
and $t$ in it~\cite{golan:semirings}.

\begin{proposition}
  $\HMod[S]$ has finite dagger biproducts. 
  When $S^+$ is zerosumfree, $\sHMod[S]$ has finite dagger biproducts.
\end{proposition}
\begin{proof}
  Let $H_1,H_2 \in \HMod[S]$ be given. Consider the $S$-semimodule
  $H = H_1 \oplus H_2$. Equip it with the inner product 
  \begin{equation}
  \label{eq:inprodbiprod}
    \inprod{h}{h'}_H = \inprod{\pi_1(h)}{\pi_1(h')}_{H_1} +
                       \inprod{\pi_2(h)}{\pi_2(h')}_{H_2}.
  \end{equation}
  Suppose that $\inprod{h}{-}_H = \inprod{h'}{-}_H$. For every $i \in
  \{1,2\}$ and $h'' \in H_i$ then 
  \[
    \inprod{\pi_i(h)}{h''}_{H_i} = \inprod{h}{\kappa_i(h'')}_H = 
    \inprod{h'}{\kappa_i(h'')}_H = \inprod{\pi_i(h')}{h''}_{H_i},
  \]
  whence $\pi_i(h)=\pi_i(h')$, and so $h=h'$. Thus $H$ is a
  Hilbert semimodule. The maps $\kappa_i$ are morphisms of $\HMod[S]$,
  as their adjoints are given by $\pi_i\colon H \to H_i$:
  \[
      \inprod{h}{\kappa_i(h')}_H
    = \inprod{\pi_1(h)}{\pi_1\kappa_i(h')}_{H_1} + 
      \inprod{\pi_2(h)}{\pi_2\kappa_i(h')}_{H_2}
    = \inprod{\pi_i(h)}{h'}_{H_i}.
  \]
  For $\sHMod[S]$ we need to verify that $H$ is strict when $H_1$ and
  $H_2$ are. Suppose $\inprod{h}{h}_H = 0$. Then
  $\inprod{\pi_1(h)}{\pi_1(h)}_{H_1} +
  \inprod{\pi_2(h)}{\pi_2(h)}_{H_2} = 0$. Since $S^+$ is zerosumfree,
  we have $\inprod{\pi_i(h)}{\pi_i(h)}_{H_i}=0$ for $i=1,2$. Hence
  $\pi_i(h)=0$, because $H_i$ is strict. Thus $h=0$, and $H$ is indeed
  strict. 
\end{proof}

\begin{proposition}
  $\HMod[S]$ is symmetric dagger monoidal. When $S$ is
  multiplicatively cancellative, $\sHMod[S]$ is symmetric
  dagger monoidal.
\end{proposition}
\begin{proof}
  Let $H,K$ be Hilbert $S$-semimodules; then $H \tensor K$ is again an
  $S$-semimodule. Define an equivalence relation $\sim$ on $H \tensor
  K$ by setting
  \[
    h \tensor k \sim h' \tensor k'
    \quad\mbox{iff}\quad 
    \inprod{h}{-}_H \cdot \inprod{k}{-}_K = 
    \inprod{h'}{-}_H \cdot \inprod{k'}{-}_K \colon  H \oplus K \to S.
  \]
  This is a congruence relation (see~\cite{golan:semirings}),
  so $H \tensor_H K = H \tensor K \slash\mathop{\sim}$ is again an
  $S$-semimodule. Defining an inner product on it by 
  \[
      \inprod{[h \tensor k]_\sim}{[h'\tensor k']_\sim}_{H \tensor_H K} 
    = \inprod{h}{h'}_H \cdot \inprod{k}{k'}_K.
  \]
  makes $H \tensor_H K$ into a Hilbert semimodule. 

  Now let $f\colon H\to H'$ and $g\colon K \to K'$ be morphisms of $\HMod[S]$. 
  Define $f \tensor_H g \colon  H \tensor_H K \to H' \tensor_H K'$ by
  $(f \tensor_H g)([h \tensor k]_\sim) = [f(h) \tensor g(k)]_\sim$. 
  This is a well-defined function, for if $h \tensor k \sim h' \tensor
  k'$, then
  \begin{align*}
        \inprod{f(h)}{-}_{H'} \cdot \inprod{g(k)}{-}_{K'}
    & = \inprod{h}{f^\dag(-)}_H \cdot \inprod{k}{g^\dag(-)}_K \\
    & = \inprod{h'}{f^\dag(-)}_H \cdot \inprod{k'}{g^\dag(-)}_K \\
    & = \inprod{f(h')}{-}_{H'} \cdot \inprod{g(k')}{-}_{K'},
  \end{align*}
  and hence $(f \tensor_H g)(h \tensor k) \sim (f \tensor_H g)(h'
  \tensor k')$. Moreover, it is adjointable, and hence a morphism of
  $\HMod[S]$:
  \begin{align*}
        \inprod{(f \tensor_H g)(h \tensor k)}{(h' \tensor k')}_{H' \tensor_H K'}
    & = \inprod{f(h) \tensor g(k)}{h' \tensor k'}_{H' \tensor_H K'} \\
    & = \inprod{f(h)}{h'}_{H'} \cdot \inprod{g(k)}{k'}_{K'} \\
    & = \inprod{h}{f^\dag(h')}_H \cdot \inprod{k}{g(k')}_K \\
    & = \inprod{h \tensor k}{f^\dag(h') \tensor g^\dag(k')}_{H \tensor_H K} \\
    & = \inprod{h \tensor k}{(f^\dag \tensor g^\dag)(h' \tensor k')}_{H \tensor_H K}
  \end{align*}
  In the same way, one shows that the coherence isomorphisms $\alpha$,
  $\lambda$, $\rho$ and $\gamma$ of the tensor product in $\SMod[S]$
  respect $\sim$, and descend to dagger isomorphisms in
  $\HMod[S]$. For example:
  \begin{align*}
        \inprod{\lambda(s \tensor h)}{h'}_H
    & = \inprod{sh}{h'}_H \\
    & = s^\ddag \inprod{h}{h'}_H \\ 
    & = \inprod{s}{1}_S \cdot \inprod{h}{h'}_H \\
    & = \inprod{s \tensor h}{1 \tensor h'}_{S \tensor_H H} \\
    & = \inprod{s \tensor h}{\lambda^{-1}(h')}_{S \tensor _H H},
  \end{align*}
  so $\lambda^\dag = \lambda^{-1}$.
  A routine check shows that $(\tensor_H,S)$ makes $\HMod[S]$ into a
  symmetric monoidal category.

  Finally, let us verify that these tensor products descend to
  $\sHMod[S]$ when $S$ is multiplicatively cancellative.
  Suppose $0 = \inprod{[h \tensor k]_\sim}{[h \tensor k]_\sim}_{H
  \tensor_H K} = \inprod{h}{h}_H \cdot \inprod{k}{k}_K$. Then since
  $S$ is multiplicatively cancellative, either $\inprod{h}{h}_H=0$ or
  $\inprod{k}{k}_H =0$. Since $H$ and $K$ are assumed strict, this
  means that either $h=0$ or $k=0$. In both cases we conclude $[h
  \tensor k]_\sim = 0$, so that $H \tensor_H K$ is indeed strict.
\end{proof}

Now suppose $\cat{H}$ is a nontrivial\footnote{The unique trivial
semiring $S$ with $0=1$ is sometimes excluded from consideration by 
convention. For example, fields usually require $0 \neq 1$ by
definition. In our case, the semiring $S$ is trivial iff
the category $\cat{H}$ is trivial, \ie when $\cat{H}$ is the
one-morphism (and hence one-object) category. For this reason many
results in this paper assume $\cat{H}$ to be nontrivial, but the main
result, Theorem~\ref{thm:main}, holds regardlessly.} locally small
pre-Hilbert category with monoidal unit $\I$. Then $S=\cat{H}(\I,\I)$
is a commutative involutive semiring, and $\cat{H}$ is enriched over
$\SMod[S]$. Explicitly, the zero
morphism is the unique one that factors through the zero object, the
sum $f+g$ of two morphisms $f,g \colon X \to Y$ is given by
\[\xymatrix@1{
  X \ar^-{\Delta}[r] & X \oplus X \ar^-{f \oplus g}[r] & Y \oplus Y
  \ar^-{\nabla}[r] & Y,
}\]
and the multiplication of a morphism $f\colon X \to Y$ with a scalar
$s \colon \I \to \I$ is determined by
\[\xymatrix@1{
  X \ar^-{\cong}[r] & \I \tensor X \ar^-{s \tensor f}[r] & \I \tensor Y
  \ar^-{\cong}[r] & Y.
}\]
The scalar multiplication works more generally for symmetric monoidal
category~\cite{abramsky:scalars}. The fact that the above provides an
enrichment in $\SMod[S]$ (and that this enrichment is furthermore
functorial) is proved in~\cite{heunen:semimoduleenrichment}.
Hence there is a functor $\cat{H}(\I,-) \colon  \cat{H} \to
\SMod[S]$. If $\I$ is a generator, this functor is faithful. We will
now show that this functor in fact factors through $\sHMod[S]$.

\begin{lemma}
\label{lem:semiring}
  Let $\cat{H}$ be a nontrivial locally small pre-Hilbert
  category. Denote by $\I$ its monoidal unit. Then $S=\cat{H}(\I,\I)$ is
  a commutative involutive semiring, and $S^+$ is zerosumfree. When
  moreover $\I$ is simple, $S$ is multiplicatively cancellative.
\end{lemma}
\begin{proof}
  For the proof that $S$ is a semiring we refer
  to~\cite{heunen:semimoduleenrichment}.  
  If $\I$ is simple, \cite[3.5]{vicary:complexnumbers} shows that $S$
  is multiplicatively cancellative, and~\cite[3.10]{vicary:complexnumbers} 
  shows that $S^+$ is zerosumfree in any case. 
\end{proof}

\begin{theorem}
\label{thm:phase1}
  Let $\cat{H}$ be a nontrivial locally small pre-Hilbert category. Denote
  its monoidal unit by $\I$. There is a functor 
  $\cat{H}(\I,-) \colon  \cat{H} \to \sHMod[S]$ for $S=\cat{H}(\I,\I)$. It
  preserves $\dag$, $\oplus$, and kernels. It is monoidal when $\I$ is simple.
  It is faithful when $\I$ is a generator. 
\end{theorem}
\begin{proof}
  We have to put an $S$-valued inner product on
  $\cat{H}(\I,X)$. Inspired by Lemma~\ref{lem:HModSXisX}, we define
  $\inprod{-}{-} \colon  \cat{H}(\I,X)^\ddag \tensor \cat{H}(\I,X) \to
  \cat{H}(\I,\I)$ by (linear extension of) $\inprod{x}{y} = x^\dag
  \after y$ for $x,y \in \cat{H}(\I,X)$. The Yoneda lemma shows that
  its power transpose $x \mapsto x^\dag \after (-)$ is a
  monomorphism. Thus $\cat{H}(\I,X)$ is a Hilbert $S$-semimodule.   
  A forteriori, \cite[2.11]{vicary:complexnumbers} shows that it is a
  strict Hilbert semimodule.

  Moreover, the image of a morphism $f\colon X \to Y$ of
  $\cat{H}$ under $\cat{H}(\I,-)$ is indeed a morphism of $\sHMod[S]$,
  that is, it is adjointable, since
  \[
       \inprod{f \after x}{y}_{\cat{H}(\I,Y)}
     = (f \after x)^\dag \after y 
     = x^\dag \after f^\dag \after y 
     = \inprod{x}{f^\dag \after y}_{\cat{H}(\I,X)}
  \]
  for $x \in \cat{H}(\I,X)$ and $y \in \cat{H}(\I,Y)$.
  This also shows that $\cat{H}(\I,-)$ preserves $\dag$. Also, by
  definition of product we have $\cat{H}(\I,X \oplus Y) \cong
  \cat{H}(\I,X) \oplus \cat{H}(\I,Y)$, so the functor $\cat{H}(\I,-)$
  preserves $\oplus$. 

  To show that $\cat{H}(\I,-)$ preserves kernels, suppose that
  $\smash{\xymatrix@1{k=\ker(f)\colon  K\; \ar@{ |>->}[r] & X}}$ is 
  a kernel of $f\colon  X \to Y$ in $\cat{H}$. We have to show that
  $\smash{\xymatrix@1{\cat{H}(\I,k) = k \after (-) \colon  \cat{H}(\I,K)\;
  \ar@{ |>->}[r] & \cat{H}(\I,X)}}$ is a kernel of $\cat{H}(\I,f) = f
  \after (-) \colon  \cat{H}(\I,X) \to \cat{H}(\I,Y)$ in
  $\sHMod[S]$.
  First of all, one indeed has $\cat{H}(\I,f) \after \cat{H}(\I,k) =
  \cat{H}(\I,f \after k) = 0$. Now suppose that $l \colon  Z \to
  \cat{H}(\I,X)$ also satisfies $\cat{H}(\I,f) \after l = 0$. That is,
  for all $z \in Z$, we have $f \after (l(z)) = 0$. Since $k$ is a
  kernel, for each $z \in Z$ there is a unique $m_z\colon  \I \to K$
  with $l(z) = k \after m_z$. Define a function $m\colon  Z \to
  \cat{H}(\I,K)$ by $m(z)=m_z$. This is a well-defined module
  morphism, since $l$ is; for example, 
  \[
      k \after m_{z+z'} 
    = l(z+z') 
    = l(z) + l(z') 
    = (k \after m_z) + (k \after m_{z'}) 
    = k \after (m_z + m_{z'}),
  \]
  so that $m(z+z')=m(z)+m(z')$ because $k$ is mono. In fact, $m$ is
  the unique module morphism satisfying $l = \cat{H}(\I,k) \after m$.
  Since $k$ is a dagger mono, we have $m = \cat{H}(\I,k^\dag) \after
  l$. So as a composition of adjointable module morphisms $m$ is a
  well-defined morphism of $\sHMod[S]$. Thus $\cat{H}(\I,k)$ is 
  indeed a kernel of $\cat{H}(\I,f)$, and $\cat{H}(\I,-)$ preserves kernels.

  If $\I$ is simple then $\sHMod[S]$ is monoidal. To show that
  $\cat{H}(\I,-)$ is a monoidal functor we must give a natural
  transformation $\varphi_{X,Y} \colon  \cat{H}(\I,X) \tensor \cat{H}(\I,Y)
  \to \cat{H}(\I, X \tensor Y)$ and a morphism $\psi \colon  S \to
  \cat{H}(\I,\I)$. Since $S=\cat{H}(\I,\I)$, we can simply take
  $\psi=\id$. Define $\varphi$ by mapping $x \tensor y$ for $x\colon \I \to
  X$ and $y\colon \I \to Y$ to 
  \[\xymatrix@1{
    \I \ar^-{\cong}[r] & \I \tensor \I \ar^-{x \tensor y}[r] & X \tensor Y.
  }\]
  It is easily seen that $\varphi$ and $\psi$ make the required
  coherence diagrams commute, and hence $\cat{H}(\I,-)$ is a monoidal
  functor. 
\end{proof}

\section{The scalar field}
\label{sec:scalarfield}

This section shows that the scalars in a pre-Hilbert category whose
monoidal unit is a simple generator necessarily form an involutive field. 
First, we need a factorisation result, which is interesting in its own right.

\begin{lemma} In any dagger category:
  \label{lem:dagmonoproperties} 
  \begin{enumerate}
  \renewcommand{\labelenumi}{\emph{(\alph{enumi})}}
    \item A dagger mono which is epic is a dagger isomorphism. 
    \item If both $gf$ and $f$ are dagger epic, so is $g$. 
    \item If $m$ and $n$ are dagger monic, and $f$ is an isomorphism with
          $nf=m$, then $f$ is a dagger isomorphism.
  \end{enumerate}
\end{lemma}
\begin{proof}
  For (a), notice that $ff^\dag=\id$ implies $ff^\dag f=f$,
  from which $f^\dag f=\id$ follows from the assumption that $f$ is epi.
  For (b): $gg^\dag = gff^\dag g = gf (gf)^\dag = \id$.
  Finally, consider (c). If $f$ is isomorphism, in particular it is epi. If
  both $nf$ and $n$ are dagger mono, then so is $f$, by (b). Hence by
  (a), $f$ is dagger isomorphism.
\end{proof}

% Since a pre-Hilbert category has dagger kernels, it automatically also has
% dagger cokernels by $\coker(f) = \ker(f^\dag)^\dag$.

\begin{lemma}
  \label{lem:monokerzero}
  In any pre-Hilbert category, a morphism $m$ is mono iff $\ker(m)=0$.
\end{lemma}
\begin{proof}
  Suppose $\ker(m)=0$. Let $u,v$ satisfy $mu=mv$. Put $q$ to be the
  dagger coequaliser of 
  $u$ and $v$. Since $q$ is dagger epic, $q=\coker(w)$ for some $w$.
  As $mu=mv$, $m$ factors through $q$ as $m=nq$.
  Then $mw=nqw=n0=0$, so $w$ factors through $\ker(m)$ as $w=\ker(m)
  \after p$ for some $p$. But since $\ker(m)=0$, $w=0$. So $q$ is a
  dagger isomorphism, and in particular mono. Hence, from $qu=qv$ follows
  $u=v$. Thus $m$ is mono.
  
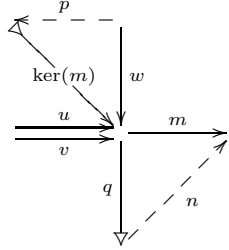
\[\xymatrix@C+3ex@R+3ex{
    \ar@{ |>->}|-{\ker(m)}[dr] & \ar@{-->}_-{p}[l] \ar^-{w}[d] \\
    \ar@<.5ex>^-{u}[r] \ar@<-.5ex>_-{v}[r] 
    & \ar^-{m}[r] \ar@{-|>}_-{q}[d] & \\
    & \ar@{-->}_-{n}[ur] 
  }\]
  Conversely, if $m$ is mono, $\ker(m)=0$ follows from $m \after \ker(m) = 0 =
  m \after 0$.
\end{proof}

\begin{lemma}
  \label{lem:factorisation}
  Any morphism in a pre-Hilbert category can be factored as a
  dagger epi followed by a mono. This factorisation is unique up to a
  unique dagger isomorphism. 
\end{lemma}
\begin{proof} 
  Let a morphism $f$ be given.
  Put $k=\ker(f)$ and $e=\coker(k)$.
  Since $fk=0$ (as $k=\ker(f)$), $f$ factors through $e$($=\coker(k)$)
  as $f=me$.
  
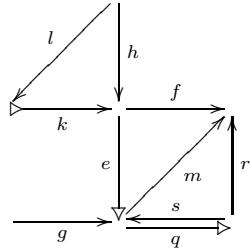
\[\xymatrix@C+3ex@R+3ex{
    & \ar^-{h}[d] \ar_-{l}[dl] \\
    \ar@{ |>->}_-{k}[r] & \ar@{-|>}_-{e}[d] \ar^-{f}[r] & \\
    \ar_-{g}[r] & \ar_-{m}[ur] \ar@{-|>}@<-.5ex>_-{q}[r] 
    & \ar_-{r}[u] \ar@<-.25ex>_-{s}[l]
  }\]
  We have to show that $m$ is mono.
  Let $g$ be such that $mg=0$. By Lemma~\ref{lem:monokerzero} 
  it suffices to show that $g=0$.
  Since $mg=0$, $m$ factors through $q=\coker(g)$ as $m=rq$.
  Now $qe$ is a dagger epi, being the composite of two dagger epis.
  So $qe=\coker(h)$ for some $h$.
  Since $fh=rqeh=r0=0$, $h$ factors through $k$($=\ker(f)$) as
  $h=kl$. 
  Finally $eh=ekl=0l=0$, so $e$ factors through $qe=\coker(h)$ as
  $q=sqe$. 
  But since $e$ is (dagger) epic, this means $sq=\id$, whence $q$ is
  mono.
  It follows from $qg=0$ that $g=0$, and the factorisation is
  established. 
  Finally, by Lemma~\ref{lem:dagmonoproperties}(c), the factorisation is
  unique up to a dagger isomorphism.
\end{proof}

We just showed that any Hilbert category has a factorisation system
consisting of monos and dagger epis. Equivalently, it has a factorisation
system of epis and dagger monos. Indeed, if we can factor $f^\dag$ as an
dagger epi followed by a mono, then taking the daggers of those, we
find that $f^{\dag\dag}=f$ factors as an epi followed by a dagger mono.
The combination of both factorisations yields that every morphism can
be written as a dagger epi, followed by a monic epimorphism, followed
by a dagger mono; this can be thought of as a generalisation of
\emph{polar decomposition}. 

Recall that a \emph{semifield} is a
commutative semiring in which every nonzero element has a
multiplicative inverse. Notice that the scalars in the embedding
theorem for Abelian categories do not necessarily have multiplicative
inverses. 

\begin{lemma}
\label{lem:scalardivision}
  If $\cat{H}$ is a nontrivial pre-Hilbert category with simple monoidal
  unit $\I$, then $S=\cat{H}(\I,\I)$ is a semifield.
\end{lemma}
\begin{proof}
  We will show that $S$ is a semifield by proving that any $s \in S$
  is either zero or isomorphism. Factorise $s$ as $s=me$ for a dagger mono
  $\xymatrix@1{m\colon \Im(s)\;\ar@{ |>->}[r] & \I}$ and an epi $e\colon \I
  \twoheadrightarrow \Im(s)$. Since $\I$ is simple, either $m$ is zero or
  $m$ is isomorphism. If $m=0$ then $s=0$. If $m$ is isomorphism, then
  $s$ is epi, so 
  $s^\dag$ is mono. Again, either $s^\dag=0$, in which case $s=0$, or
  $s^\dag$ is isomorphism. In this last case $s$ is also isomorphism.
\end{proof}

The following lemma shows that every scalar also has an additive
inverse. This is always the case for the scalars in the embedding
theorem for Abelian categories, but the usual proof of this fact is
denied to us because epic monomorphisms are not necessarily
isomorphisms in a pre-Hilbert category (see
Appendix~\ref{sec:thecategoryHilb}). 

\begin{lemma}
\label{lem:scalarsubtraction}
  If $\cat{H}$ is a nontrivial pre-Hilbert category whose monoidal
  unit $\I$ is a simple generator, then $S=\cat{H}(\I,\I)$ is a field.
\end{lemma}
\begin{proof}
  Applying~\cite[4.34]{golan:semirings} to the previous lemma yields
  that $S$ is either zerosumfree, or a field. Assume, towards a
  contradiction, that $S$ is zerosumfree. We will show that the kernel
  of the codiagonal $\nabla = [\id,\id] \colon  \I \oplus \I \to \I$ is zero.
  Suppose $\nabla \after \langle x,y \rangle = x+y = 0$ for $x,y\colon X \to
  \I$. Then for all $z\colon \I \to X$ we have $\nabla \after \langle x,y
  \rangle \after z = 0 \after z = 0$, \ie $xz+yz = 0$. Since $S$ is
  assumed zerosumfree hence $xz=yz=0$, so $\langle x,y\rangle \after z
  = 0$. Because $\I$ is a generator then $\langle x,y \rangle = 0$.
  Thus $\ker(\nabla)=0$. But then, by Lemma~\ref{lem:monokerzero},
  $\nabla$ is mono, whence $\kappa_1 = \kappa_2$, which is a
  contradiction. 
\end{proof}

Collecting the previous results about the scalars in a pre-Hilbert
category yields Theorem~\ref{thm:scalars} below. It uses a well-known
characterisation of subfields of the complex numbers, that we recall
in the following two lemmas. 

\begin{lemma}
\label{lem:embeddinginalgebraicclosure}
  \cite[Theorem~4.4]{grillet:abstractalgebra}
  Any field of characteristic zero and at most continuum cardinality
  can be embedded in an algebraically closed field of characteristic
  zero and continuum cardinality.
  \qed
\end{lemma}

\begin{lemma}
\label{lem:algebraicallyclosedfieldscategorical}
  \cite[Proposition~1.4.10]{changkeisler:modeltheory}
  All algebraically closed fields of characteristic zero and
  continuum cardinality are isomorphic.
  \qed
\end{lemma}

\begin{theorem}
\label{thm:scalars}
  If $\cat{H}$ is a nontrivial pre-Hilbert category whose monoidal
  unit $\I$ is a simple generator, then $S=\cat{H}(\I,\I)$ is an
  involutive field of characteristic zero of at most continuum
  cardinality, with $S^+$ zerosumfree.

  Consequently, there is a monomorphism $\cat{H}(\I,\I)
  \rightarrowtail \field{C}$ of fields. However, it does not
  necessarily preserve the involution. 
\end{theorem}
\begin{proof}
  To establish characteristic zero, we have to prove that for all
  scalars $s \colon \I \to \I$ 
  the property $s+\cdots+s=0$ implies $s=0$, where the sum contains
  $n$ copies of $s$, for all $n \in \{1,2,3,\ldots\}$.. So suppose
  that $s + \cdots + s = 0$. 
  By definition, $s+ \cdots + s = \nabla^n \after (s \oplus \cdots
  \oplus s) \after \Delta^n = \nabla^n \after \Delta^n \after s$,
  where $\nabla^n = [\id]_{i=1}^n\colon  \bigoplus_{i=1}^n \I \to \I$ and 
  $\Delta^n = \langle\id\rangle_{i=1}^n \colon  \I \to
  \bigoplus_{i=1}^n \I$ are the $n$-fold (co)diagonals. But $0 \neq
  \nabla^n \after \Delta^n = (\Delta^n)^\dag \after \Delta^n$ by Lemma
  2.11 of \cite{vicary:complexnumbers}, which states that $x^\dag x=0$
  implies $x=0$ for every $x \colon \I \to X$. Since $S$ is a field by
  Lemma~\ref{lem:scalarsubtraction}, this means that $s=0$. 
\end{proof}

This theorem is of interest to reconstruction programmes,
that try to derive major results of quantum theory from simpler
mathematical assumptions, for among the things to be reconstructed are 
the scalars. For example, \cite{soler} shows that if an orthomodular
pre-Hilbert space is infinite dimensional, then the base field is
either $\field{R}$ or $\field{C}$, and the space is a Hilbert space. 

With a scalar field, we can sharpen the preservation of finite
biproducts and kernels of Theorem~\ref{thm:phase1} to preservation of
all finite limits. Since $\cat{H}(\I,-)$ preserves the dagger, it hence
also preserves all finite colimits. (In other terms:
$\cat{H}(\I,-)$ is exact.) 

\begin{corollary}
\label{col:diagramchasing}
  The functor $\cat{H}(\I,-) \colon  \cat{H} \to \sHMod[\cat{H}(\I,\I)]$
  preserves all finite limits and all finite colimits, for any
  pre-Hilbert category $\cat{H}$ whose monoidal unit $\I$ is a simple
  generator.
\end{corollary}
\begin{proof}
  One easily checks that $F=\cat{H}(\I,-)$ is an $\Cat{Ab}$-functor, \ie
  that $(f+g) \after (-) = (f \after -) + (g \after
  -)$~\cite{heunen:semimoduleenrichment}. %
  Hence, $F$ preserves equalisers:
  \[
      F(\mathrm{eq}(f,g))
    = F(\ker(f-g))
    = \ker(F(f-g))
    = \ker(Ff-Fg)
    = \mathrm{eq}(Ff,Fg).
  \]
  Since we already know from Theorem~\ref{thm:phase1} that $F$
  preserves finite products, we can conclude that it preserves all 
  finite limits. And because $F$ preserves the self-duality $\dag$, it
  also preserves all finite colimits.
\end{proof}

\section{Extension of scalars}
\label{sec:Extensionofscalars}

The main idea underlying this section is to exploit
Theorem~\ref{thm:scalars}. We will construct a functor $\HMod[R] \to
\HMod[S]$ given a morphism $R \to S$ of commutative involutive
semirings, and apply it to the above $\cat{H}(\I,\I) \to
\field{C}$. This is called \emph{extension of scalars}, and is well
known in the setting of modules (see \eg
\cite[10.8.8]{ash:abstractalgebra}). Let us first consider in some
more detail the construction on semimodules.  

Let $R$ and $S$ be commutative semirings, and $f\colon R \to S$ a homomorphism of
semirings. Then any $S$-semimodule $M$ can be considered an
$R$-semimodule $M_R$ by defining scalar multiplication $r\cdot m$ in
$M_R$ in terms of scalar multiplication of $M$ by $f(r) \cdot m$.
In particular, we can regard $S$ as an $R$-semimodule. 
Hence it makes sense to look at $S \tensor_R M$. Somewhat more
precisely, we can view $S$ as a left-$S$-right-$R$-bisemimodule, and $M$
as a left-$R$-semimodule. Hence $S \tensor_R M$ becomes a
left-$S$-semimodule (see~\cite{golan:semirings}). This construction
induces a functor $f^* \colon  \SMod[R] \to \SMod[S]$, acting on morphisms
$g$ as $\id \tensor_R g$. It is easily seen to be strong monoidal and to
preserve biproducts and kernels. 

Now let us change to the setting where $R$ and $S$ are involutive
semirings, $f\colon R \to S$ is a morphism of involutive semirings, and we
consider Hilbert semimodules instead of semimodules. The next theorem
shows that this construction lifts to a functor $f^*\colon \sHMod[R]
\to \sHMod[S]$ (under some conditions on $S$ and
$f$). Moreover, the fact 
that any $S$-semimodule can be seen as an $R$-semimodule via $f$
immediately induces another functor $f_* \colon  \SMod[S] \to
\SMod[R]$. This one is called \emph{restriction of scalars}. In fact,
$f_*$ is right adjoint to $f^*$~\cite[vol 1, 3.1.6e]{borceux}. 
% \[
%   \infer={\xymatrix@1{M \ar_-{h}[r] & N \qquad \mbox{(in $\SMod[R]$)}}} 
%          {\xymatrix@1{S \tensor_R M \ar^-{g}[r] & N \quad \mbox{(in $\SMod[S]$)}}}
% \]
% The transpose of $g$ is defined by $g^\vee(m) = g(1_S \tensor m)$,
% and the transpose of $h$ is defined by $h^\wedge(s \tensor m) = s
% \cdot h(m)$. These are well-defined morphisms such that
% $g^{\vee\wedge}=g$ and $h^{\wedge\vee}=h$, and moreover are natural in
% $M$.
However, since we do not know
how to fashion an sesquilinear $R$-valued form out of an $S$-valued
one in general, it seems impossible to construct an adjoint functor
$f_*\colon \sHMod[S] \to \sHMod[R]$. 

\begin{theorem}
\label{thm:phase2}
  Let $R$ be a commutative involutive semiring, $S$ be a
  multiplicatively cancellative commutative involutive ring, and $f\colon R
  \rightarrowtail S$ be a monomorphism of involutive semirings. There
  is a faithful functor $f^* \colon  \sHMod[R] \to \sHMod[S]$ that preserves
  $\dag$. 
  If $R$ is multiplicatively cancellative, then $f^*$ is strong
  monoidal. If both $R^+$ and $S^+$ are zerosumfree, then $f^*$ also
  preserves $\oplus$. 
\end{theorem}
\begin{proof}
  Let $M$ be a strict Hilbert $R$-semimodule. Defining the carrier of
  $f^*M$ to be $S \tensor_R M$ turns it into an $S$-semimodule as 
  before. Furnish it with
  \[
    \inprod{s \tensor m}{s' \tensor m'}_{f^*M} = s^\ddag \cdot s'
    \cdot f(\inprod{m}{m'}_M).
  \]
  Assume $0=\inprod{s \tensor m}{s \tensor m}_{f^*M} = s^\ddag s
  f(\inprod{m}{m}_M)$. Since $S$ is multiplicatively cancellative,
  either $s=0$ or $f(\inprod{m}{m}_M)=0$. In the former case
  immediately $s \tensor m=0$. In the latter case $\inprod{m}{m}_M=0$
  since $f$ is injective, and because $M$ is strict $m=0$, whence $s
  \tensor m=0$. Since $S$ is a ring, this implies that $f^*M$ is a
  strict Hilbert $S$-semimodule. For if $\inprod{x}{-}_{f^*M} =
  \inprod{y}{-}_{f^*M}$ then $\inprod{x-y}{-}_{f^*M} = 0$, so in
  particular $\inprod{x-y}{x-y}_{f^*M}=0$. Hence $x-y=0$ and $x=y$.
  
  Moreover, the image of a morphism $g\colon M \to M'$ of $\sHMod[R]$ under
  $f^*$ is a morphism of $\sHMod[S]$, as its adjoint is $\id \tensor
  g^\dag$: 
  \begin{align*}
        \inprod{(\id \tensor g)(s \tensor m)}{s' \tensor m'}_{f^*M'}
    & = \inprod{s \tensor g(m)}{s' \tensor m'}_{f^*M'} \\
    & = s^\ddag s' f(\inprod{g(m)}{m'}_M') \\
    & = s^\ddag s' f(\inprod{m}{g^\dag(m')}_M) \\
    & = \inprod{s \tensor m}{s' \tensor g^\dag(m')}_{f^*M} \\
    & = \inprod{s \tensor m}{(\id \tensor g^\dag)(s' \tensor m')}_{f^*M}.
  \end{align*}
  Obviously, $f^*$ is faithful, and preserves $\dag$.
%   To show that $f^*$ preserves kernels, we use the explicit form of
%   kernels in a category of (Hilbert) semimodules, namely that
%   $\ker(f^*(g)) = \{ \varphi \in S \tensor_R X \mid
%   (\id \tensor g)(\varphi)=0 \}$ and $f^*(\ker(g)) = S
%   \tensor_R \{x \in X \mid g(x)=0\}$ for $g\colon X \to Y$.
%   To verify that these coincide it suffices to observe that $1 \tensor
%   x \in S \tensor_R X = 0$ iff $x=0$ for $x \in X$.
  If dagger biproducts are available, then $f^*$ preserves them, since 
  biproducts distribute over tensor products.
  If dagger tensor products are available, showing that $f^*$
  preserves them comes down to giving an isomorphism
  $S \to S \tensor_R R$ and a natural isomorphism 
  $(S \tensor_R X) \tensor_S (S \tensor_R Y)
  \to S \tensor_R (X \tensor_R Y)$. The obvious candidates for
  these satisfy the coherence diagrams, making $f^*$ strong monoidal.
\end{proof}

\begin{corollary}
\label{cor:phase2}
  Let $S$ be an involutive field of characteristic zero and at most
  continuum cardinality. Then there is a strong monoidal, faithful
  functor $\sHMod[S] \to \sHMod[\field{C}]$ that preserves all finite
  limits and finite colimits, and preserves $\dag$ up to an
  isomorphism of the base field.
\end{corollary}
\begin{proof}
  The only claim that does not follow from previous results is the
  statement about preservation of finite (co)limits. This comes down
  to a calculation in the well-studied situation of module
  theory~\cite[Exercise~10.8.5]{ash:abstractalgebra}. 
\end{proof}

Note that the extension of scalars functor $f^*$ of the previous
theorem is full if and only if $f$ is a regular epimorphism, \ie iff
$f$ is surjective. To see this, consider the inclusion
$f\colon \field{N} \hookrightarrow \field{Z}$. This is obviously not
surjective. Now, $\SMod[\field{N}]$ can be identified with the
category $\Cat{cMon}$ of commutative monoids, and $\SMod[\field{Z}]$
can be identified with the category $\Cat{Ab}$ of Abelian
groups. Under this identification, $f^*\colon \Cat{cMon} \to \Cat{Ab}$ sends
an object $X \in \Cat{cMon}$ to $X \coprod X$, with inverses being
provided by swapping the two terms $X$. For a morphism $g$, $f^*(g)$
sends $(x,x')$ to $(gx,gx')$. Consider $h\colon X \coprod X \to X \coprod
X$, determined by $h(x,x')=(x',x)$. If $h=f^*(g)$ for some $g$, then
$(x',x) = h(x,x') = f^*g(x,x') = (gx,gx')$, so $gx=x'$ and $gx'=x$ for
all $x,x' \in X$. Hence $g$ must be constant, contradicting $h=f^*g$.
Hence $f^*$ is not full.

\section{Completion}
\label{sec:completion}

Up to now we have concerned ourselves with algebraic structure
only. To arrive at the category of Hilbert spaces and continuous
linear maps, some analysis comes into play.
Looking back at Definition~\ref{def:hilbertsemimodule}, we see that a
strict Hilbert $\field{C}$-semimodule is just a pre-Hilbert space, 
\ie a complex vector space with a positive definite sesquilinear form
on it. Any pre-Hilbert space $X$ can be completed to a Hilbert space
$\widehat{X}$ into which it densely
embeds~\cite[I.3]{reedsimon:functionalanalysis}.  

A morphism $g\colon X \to Y$ of $\sHMod[\field{C}]$ amounts to a 
linear transformation between pre-Hilbert spaces that has an adjoint. 
So $\sHMod[\field{C}]=\preHilb$. However, these morphisms need not
necessarily be bounded, and only bounded adjointable morphisms can be
extended to the completion $\widehat{X}$ of their
domain~\cite[I.7]{reedsimon:functionalanalysis}. Therefore, we impose
another axiom on the morphisms of $\cat{H}$ to ensure this. Basically,
we rephrase the usual definition of boundedness of a function between
Banach spaces for morphism between Hilbert semimodules. Recall from
Lemma~\ref{thm:scalars} that the scalars $S=\cat{H}(\I,\I)$ in a
pre-Hilbert category $\cat{H}$ are always an involutive field, with
$S^+$ zerosumfree. Hence $S$ is ordered by $r \leq s$ iff $r+p=s$ for
some $p \in S^+$. We use this ordering to define boundedness of a
morphism in $\cat{H}$, together with the norm induced by the canonical
bilinear form $\inprod{f}{g} = f^\dag \after g$.

\begin{definition}
  Let $\cat{H}$ be a symmetric dagger monoidal dagger category with
  dagger biproducts. 
  A scalar $M\colon \I \to \I$ is said to \emph{bound} a morphism
  $g\colon X \to Y$ when $x^\dag g^\dag g x \leq M^\dag x^\dag x M$
  for all $x\colon \I \to X$. A morphism is called \emph{bounded} when
  it has a bound.  
  A \emph{Hilbert category} is a pre-Hilbert category whose
  morphisms are bounded.
\end{definition}

In particular, a morphism $g\colon X \to Y$ in $\sHMod[S]$ is bounded when
there is an $M \in S$ satisfying $\inprod{g(x)}{g(x)} \leq M^\dag M 
\inprod{x}{x}$ for all $x \in X$.

Almost by definition, the functor $\cat{H}(\I,-)$ of
Theorem~\ref{thm:phase1} preserves boundedness of morphisms when
$\cat{H}$ is a Hilbert category. The following lemma shows that also
the extension of scalars of Theorem~\ref{thm:phase2} preserves
boundedness. It is noteworthy that a combinatorial condition
(boundedness) on the category $\cat{H}$ ensures an analytic property 
(continuity) of its image in $\sHMod[\field{C}]$, as we never even
assumed a topology on the scalar field, let alone assuming
completeness. 

\begin{lemma}
\label{lem:extensionofscalarsbounded}
  Let $f\colon R \to S$ be a morphism of involutive semirings.
  If $g\colon X \to Y$ is bounded in $\sHMod[R]$, then $f^*(g)$ is bounded in
  $\sHMod[S]$. 
\end{lemma}
\begin{proof}
  First, notice that $f\colon R \to S$ preserves the canonical order:
  if $r \leq r'$, say $r+t^\ddag t = r'$ for $r,r',t \in R$, 
  then $f(r)+f(t)^\dag f(t) = f(r+t^\dag t) = f(r')$, so $f(r) \leq
  f(r')$.

  Suppose $\inprod{g(x)}{g(x)}_Y \leq M^\ddag M \inprod{x}{x}_X$ for
  all $x \in X$ and some $M \in R$. 
  Then $f(\inprod{g(x)}{g(x)}_Y) \leq f(M^\ddag M \inprod{x}{x}_X) =
  f(M)^\ddag f(M) f(\inprod{x}{x}_X)$ for $x \in X$.
  Hence for $s \in S$:
  \begin{align*}
        \inprod{f^*g(s \tensor x)}{f^*g (s \tensor x)}_{f^* Y}
    & = \inprod{(\id \tensor g)(s \tensor x)}
               {(\id \tensor g)(s \tensor x)}_{f^* Y} \\
    & = \inprod{s \tensor g(x)}{s \tensor g(x)}_{f^* Y} \\
    & = s^\ddag s f(\inprod{g(x)}{g(x)}_Y) \\
    & \leq s^\ddag s f(M)^\ddag f(M) f(\inprod{x}{x}_X) \\
    & = f(M)^\ddag f(M) \inprod{s \tensor x}{s \tensor x}_{f^* X}.
  \end{align*}
  Because elements of the form $s \tensor x$ form a basis for $f^* X =
  S \tensor_R X$, we thus have 
  \[
         \inprod{f^*g(z)}{f^*g(z)}_{f^* Y} 
    \leq f(M)^\ddag f(M) \inprod{z}{z}_{f^* X}
  \]
  for all $z \in f^* X$. In other words: $f^* g$ is bounded (namely,
  by  $f(M)$).
\end{proof}

Combining this section with Theorems~\ref{thm:phase1}
and~\ref{thm:scalars}, Corollary~\ref{cor:phase2} and
Lemma~\ref{lem:extensionofscalarsbounded} now results in our main
theorem. Notice that the completion preserves biproducts and kernels
and thus equalisers, and so preserves all finite limits and colimits.

\begin{theorem}
\label{thm:main}
  Any locally small Hilbert category $\cat{H}$ whose monoidal unit is
  a simple generator has a monoidal embedding into the category
  $\Hilb$ of Hilbert spaces and continuous linear maps that preserves
  $\dag$ (up to an isomorphism of the base field) and all finite limits and
  finite colimits. 
\end{theorem}
\begin{proof}
  The only thing left to prove is the case that $\cat{H}$ is trivial.
  But if $\cat{H}$ is a one-morphism Hilbert category, its one
  object must be the zero object, and its one morphism must be the
  zero morphism. Hence sending this to the zero-dimensional Hilbert
  space yields a faithful monoidal functor that preserves $\dag$ and
  $\oplus$, trivially preserving all (co)limits. 
\end{proof}

To finish, notice that the embedding of the Hilbert category
$\Cat{Hilb}$ into itself thus constructed is (isomorphic to)
the identity functor.

\section{Conclusion}
\label{sec:conclusion}

Let us conclude by discussing several further issues.

\subsection{Dimension}
\label{subsec:dimension}

The embedding of Theorem~\ref{thm:main} is strong monoidal (\ie
preserves $\tensor$) if the canonical (coherent) morphism is an
isomorphism 
\[
  \cat{H}(\I,X) \tensor \cat{H}(\I,Y) \cong \cat{H}(\I,X \tensor Y),
\]
where the tensor product in the left-hand side is that of (strict)
Hilbert semimodules. This is a quite natural restriction, as it
prevents degenerate cases like $\tensor=\oplus$.
Under this condition, the embedding preserves compact
objects~\cite{heunen:semimoduleenrichment}. This means that compact
objects correspond to finite-dimensional Hilbert spaces under the
embedding in question. Our embedding theorem also shows that every
Hilbert category embeds into a
C*-category~\cite{ghezlimaroberts:wstarcategories}. This relates to
representation theory. Compare \eg \cite{doplicherroberts:duality},
who establish a correspondence between a compact group and its
categories of finite-dimensional, continuous, unitary representations;
the latter category is characterised by axioms comparable to
those of pre-Hilbert categories, with moreover every object being
compact. 

Corollary~\ref{col:diagramchasing} opens the way to \emph{diagram
chasing} (see \eg \cite[vol 2, Section~1.9]{borceux}): to prove that a
diagram commutes in a pre-Hilbert category, it suffices to prove this
in pre-Hilbert spaces, where one has access to actual elements. 
As discussed above, when $\cat{H}$ is compact, and the embedding
$\cat{H} \to \Cat{\preHilb}$ is strong monoidal, then the embedding takes
values in the category of finite-dimensional pre-Hilbert spaces. The
latter coincides with the category of finite-dimensional Hilbert
spaces (since every finite-dimensional pre-Hilbert space is Cauchy
complete). This partly explains the main claim 
in~\cite{selinger:diagramchasing}, namely that an equation holds in 
all dagger traced symmetric monoidal categories if and only if it
holds in finite-dimensional Hilbert spaces. 

\subsection{Functor categories}

We have used the assumption that the monoidal unit is simple in an
essential way. But if $\cat{H}$ is a pre-Hilbert category whose
monoidal unit is simple, and $\cat{C}$ is any nontrivial small
category, then the functor category $[\cat{C},\cat{H}]$ is a
pre-Hilbert category, albeit one whose monoidal unit is not simple
anymore. Perhaps the embedding theorem can be extended to this
example. The conjecture would be that any pre-Hilbert category whose
monoidal unit is a generator (but not necessarily simple), embeds into
a functor category $[\cat{C},\preHilb]$ for some category
$\cat{C}$. This requires reconstructing $\cat{C}$ from $\Sub(\I)$.

% Another route one could try is the following. This paper shows that
% every commutative Hilbert category whose monoidal unit is a simple
% generator embeds into Hilbert C*-modules over a simple commutative
% C*-algebra (there is only one simple commutative C*-algebra, namely
% $\field{C}$). This might be generalised to an embedding into Hilbert
% C*-modules of any Hilbert category whose monoidal unit is a (not
% necessarily simple) generator.

Likewise, it would be preferable to be able to drop the condition that
the monoidal unit be a generator. To accomplish this, one would need
to find a dagger preserving embedding of a given pre-Hilbert category
into a pre-Hilbert category with a finite set of generators. In the
Abelian case, this can be done by moving from $\cat{C}$ to
$[\cat{C},\Cat{Ab}]$, in which $\coprod_{X \in \cat{C}} \cat{C}(X,-)$
is a generator. But in the setting of Hilbert categories there is no
analogon of $\Cat{Ab}$. Also, Hilbert categories tend not to have
infinite coproducts. 

% A related matter concerns projective Hilbert
% spaces: \cite{coecke:projectiveaxiomatics} calls a dagger category 
% \emph{projective} when $ff^\dag=gg^\dag$ implies $f=g$ for all
% $f,g\colon \I\to X$. Perhaps this can be related to the monoidal unit being a
% generator, and perhaps the embedding theorem can be extended to map
% into projective Hilbert spaces.

\subsection{Topology}

Our axiomatisation allowed inner product spaces over $\field{Q}$
as a (pre-) Hilbert category. Additional axioms, enforcing the base
field to be (Cauchy) complete and hence (isomorphic to) the real or complex
numbers, could perhaps play a role in topologising the above to yield
an embedding into sheaves of Hilbert spaces. A forthcoming paper
studies subobjects in a (pre-)Hilbert category, showing that
quantum logic is just an incarnation of categorical logic. But this is
also interesting in relation to \cite{amemiyaaraki:completeness},
which shows that a pre-Hilbert space is complete if and only if its
lattice of closed subspaces is orthomodular.

\subsection{Fullness}

A natural question is under what conditions the embedding is full.
Imitating the answer for the embedding of Abelian categories, we can
only obtain the following partial result, since Hilbert categories
need not have infinite coproducts, as opposed to $\Cat{Ab}$. 
An object $X$ in a pre-Hilbert category $\cat{H}$ with monoidal unit
$\I$ is said to be \emph{finitely generated} when there is a dagger epi
$\xymatrix@1{\bigoplus_{i \in I} \I \ar@{-|>}[r] & \;X}$ for some
finite set $I$. 

\begin{theorem}
  The embedding of Theorem~\ref{thm:phase1} is full when every object
  in $\cat{H}$ is finitely generated. 
\end{theorem}
\begin{proof}
  We have to prove that $\cat{H}(\I,-)$'s action on
  morphisms, which we temporarily denote $T\colon \cat{H}(X,Y) \to
  \sHMod[S](\cat{H}(\I,X),\cat{H}(\I,Y))$, 
  is surjective when $X$ is finitely generated. Let $\Phi \colon 
  \cat{H}(\I,X) \to \cat{H}(\I,Y)$ in $\sHMod[S]$. We must find $\varphi
  \colon  X \to Y$ in $\cat{H}$ such that $\Phi(x) = \varphi \after x$ for
  all $x\colon \I \to X$ in $\cat{H}$. 
  Suppose first that $X=\I$. Then $\Phi(x) = \Phi(\id[\I] \after x) = 
  \Phi(\id[\I]) \after x$ since $\Phi$ is a morphism of
  $S$-semimodules. So $\varphi=\Phi(\id[\I])$ satisfies $\Phi(x) =
  \varphi \after x$ for all $x \colon  \I \to X$ in $\cat{H}$.
  In general, if $X$ is finitely generated, there is a finite set $I$
  and a dagger epi $\xymatrix@1{p \colon  \bigoplus_{i \in I} \I \ar@{-|>}[r]
    & \; X}$. Denote by $\Phi_i$ the composite morphism
  \[\xymatrix{
      \cat{H}(\I,\I) \ar^-{T(\kappa_i)}[r] 
    & \cat{H}(\I, \bigoplus_{i \in I} \I) \ar^-{T(p)}[r]
    & \cat{H}(\I,X) \ar^-{\Phi}[r]
    & \cat{H}(\I,Y) \quad \mbox{ in }\sHMod[S].
  }\]
  By the previous case ($X=\I$), for each $i \in I$ there is
  $\varphi_i \in \cat{H}(\I,Y)$ such that $\Phi_i(x) = \varphi_i \after
  x$ for all $x \in S$. Define $\bar{\varphi} = [\varphi_i]_{i \in I}
  \colon  \bigoplus_{i \in I} \I \to Y$, and $\bar{\Phi} = \Phi \after T(p) \colon 
  \cat{H}(\I, \bigoplus_{i \in I} \I) \to \cat{H}(\I,Y)$. Then, for $x \in
  \cat{H}(\I, \bigoplus_{i \in I} \I)$:
  \begin{align*}
        \bar{\Phi}(x) 
    & = \Phi(p \after x) 
      = \Phi(p \after (\sum_{i \in I} \kappa_i \after \pi_i) \after x) 
      = \sum_{i \in I} \Phi(p \after \kappa_i \after \pi_i \after x) \\
    & = \sum_{i \in I} \Phi_i(\pi_i \after x) 
      = \sum_{i \in I} \varphi_i \after \pi_i \after x 
      = \bar{\varphi} \after x.
  \end{align*}
  Since $p$ is a dagger epi, it is a cokernel, say $p=\coker(f)$.
  Now
  \[
      \bar{\varphi} \after f 
    = \bar{\Phi}(f)
    = \Phi(p \after f)
    = \Phi(0)
    = 0,
  \]
  so there is a (unique) $\varphi\colon  X \to Y$ with $\bar{\varphi} =
  \varphi \after p$. Finally, for $x\colon G \to X$,
  \[
      \Phi(x)
    = \Phi(p \after p^\dag \after x)
    = \bar{\Phi}(p^\dag \after x)
    = \bar{\varphi} \after p^\dag \after x
    = \varphi \after p \after p^\dag \after x
    = \varphi \after x.
  \]
\end{proof}

\appendix
\section{The category of Hilbert spaces}
\label{sec:thecategoryHilb}

We denote the category of Hilbert spaces and continuous linear
transformations by $\Hilb$. First, we show that $\Hilb$ is actually a  
Hilbert category. Subsequently, we prove that it is not an Abelian
category. 

First, there is a dagger in $\Hilb$, by the Riesz
representation theorem. The dagger of a morphism
$f\colon X \to Y$ is its adjoint, \ie the unique $f^\dag$ satisfying
\[
  \inprod{f(x)}{y}_Y = \inprod{x}{f^\dag(y)}_X.
\]

It is also well-known that $\Hilb$ has finite dagger biproducts:
$X \oplus Y$ is carried by the direct sum of the underlying vector
spaces, with inner product
\[
  \inprod{(x,y)}{(x',y')}_{X \oplus Y} = \inprod{x}{x'}_X + \inprod{y}{y'}_Y.
\]

Furthermore, $\Hilb$ has kernels: the kernel of $f\colon X \to Y$ is
(the inclusion of) $\{ x \in X \mid f(x)=0 \}$. Since $\ker(f)$ is in
fact a closed subspace, its inclusion is isometric. That is,
$\Hilb$ in fact has dagger kernels. Consequently $\ker(g-f)$ is
a dagger equaliser of $f$ and $g$ in $\Cat{Hilb}$.

We now turn to the requirement that every dagger mono be a dagger kernel.

\begin{lemma}
\label{lem:monos}
  The monomorphisms in $\Hilb$ are the injective continuous
  linear transformations.
\end{lemma}
\begin{proof}
  If $m$ is injective, then it is obviously mono. Conversely, suppose
  that $m\colon X \rightarrowtail Y$ is mono. Let  
  $x,x' \in X$ satisfy $m(x)=m(x')$. Define $f\colon \field{C} \to X$ by
  (continuous linear extension of) $f(1)=x$, and $g\colon \field{C} \to X$
  by (continuous linear extension of) $g(1)=x'$. Then $mf=mg$, whence
  $f=g$ and $x=x'$. Hence $m$ is injective. 
\end{proof}

Recall that Hilbert spaces have orthogonal projections, that is: 
if $X$ is a Hilbert space, and $U \subseteq X$ a closed subspace,
then every $x \in X$ can be written as $x=u+u'$ for unique $u \in U$
and $u' \in U^\perp$, where
\begin{equation}
\label{eq:orthogonalsubspace}
  U^\perp = \{ x \in X \mid \forall_{u \in U} . \inprod{u}{x}=0 \}.
\end{equation}
The function that assigns to $x$ the above unique $u$ is a morphism $X
\to U$, the \emph{orthogonal projection} of $X$ onto its closed
subspace $U$. 

\begin{proposition}
  \label{prop:dagmonosarekernels}
  In $Hilb$, every dagger mono is a dagger kernel.
\end{proposition}
\begin{proof}
  Let $m\colon M \rightarrowtail X$ be a dagger mono. 
  In particular, $m$ is a split mono, and hence its image is
  closed~\cite[4.5.2]{aubin:functionalanalysis}.
  So, without loss of generality, we can assume that $m$ is the
  inclusion of a closed subspace $M \subseteq X$. 
  But then $m$ is the dagger kernel of the orthogonal projection of
  $X$ onto $M$.
\end{proof}

All in all, $\Hilb$ is a Hilbert category. So is its full
subcategory $\Cat{fdHilb}$ of finite-dimensional Hilbert categories.
Also, if $\cat{C}$ is a small category and $\cat{H}$ a
Hilbert category, then $[\cat{C}, \cat{H}]$ is again a Hilbert category.

Since $\Hilb$ has biproducts, kernels and cokernels, it
is a pre-Abelian category. But the behaviour of epis
prevents it from being an Abelian category.

\begin{lemma}
\label{lem:epis}
  The epimorphisms in $\Hilb$ are the continuous linear
  transformations with dense image.
\end{lemma}
\begin{proof}
  Let $e\colon X \to Y$ satisfy $\overline{e(X)}=Y$, and $f,g\colon Y \to Z$
  satisfy $fe=ge$. Let $y \in Y$, say $y = \lim_n e(x_n)$. Then
  \[
    f(y) = f(\lim_n e(x_n)) = \lim_n f(e(x_n)) = \lim_n g(e(x_n)) =
    g(\lim_n e(x_n)) = g(y).
  \]
  So $f=g$, whence $e$ is epi.
  
  Conversely, suppose that $e\colon X \twoheadrightarrow Y$ is epi. Then
  $\overline{e(X)}$ is a closed subspace of $Y$, so that $Y /
  \overline{e(X)}$ is again a Hilbert space, and the projection $p\colon Y
  \to Y / \overline{e(X)}$ is continuous and linear. Consider also
  $q\colon Y \to Y / \overline{e(X)}$ defined by $q(y)=0$. Then $pe=qe$, 
  whence $p=q$, and $\overline{e(X)}=Y$.
\end{proof}

From this, we can conclude that $\Hilb$ is not an Abelian
category, since it is not balanced: there are monic epimorphisms that
are not isomorphic. In other words, there are injections that have dense
image but are not surjective. For example, $f\colon \ell^2(\field{N}) \to
\ell^2(\field{N})$ defined by $f(e_n) = \frac{1}{n}e_n$ is not
surjective, as $\sum_n \frac{1}{n}e_n$ is not in its range. But it is
injective, self-adjoint, and hence also has dense image.

Another way to see that $\Hilb$ is not an Abelian category is to
assert that the inclusion of a nonclosed subspace is mono, but 
cannot be a kernel since these are closed.

\end{document}